\documentclass[12pt]{amsart}

\usepackage[margin=1in,letterpaper,portrait]{geometry}
\usepackage{hyperref}
\usepackage{mathtools}
\usepackage{stmaryrd}
\usepackage{xstring}
\usepackage{tikz}
\usepackage{xcolor}

\usetikzlibrary{patterns}

\usetikzlibrary{arrows,shapes,automata,backgrounds,decorations,petri,positioning}
\tikzstyle{edge} = [fill,opacity=.5,fill opacity=.5,line cap=round, line join=round, line width=50pt]

\theoremstyle{definition}
\newtheorem{proposition}{Proposition}[section]
\newtheorem{lemma}[proposition]{Lemma}

\newtheorem{theorem}[proposition]{Theorem}
\newtheorem{observation}[proposition]{Observation}
\newtheorem{definition}[proposition]{Definition}
\newtheorem{example}[proposition]{Example}
\newtheorem{corollary}[proposition]{Corollary}

\newcommand{\Count}[1]{[#1]}
\newcommand{\PatternCount}[2]{\Count{#1}({#2})} 

\newcommand{\BondedTerms}[1]{\underbracket[0.3pt][0.8pt]{#1}}
\newcommand{\VinPat}[1]{
  \foreach \g in {#1} {
    \BondedTerms{\vphantom{x}\smash{\g}}
  }
}
\newcommand{\VinCount}[1]{\Count{\VinPat{#1}}}

\newcommand{\AVinP}[2]{
  \underset{
    {\scriptscriptstyle{ \foreach \a/\b in {#2} { \a \shortrightarrow \b } }}
  }{\VinPat{#1}}
}

\newcommand{\AVinC}[2]{\Count{\AVinP{#1}{#2}}}

\newcommand{\identity}{e}

\newcommand{\MeshPattern}[4]{
  \raisebox{0.6ex}{
  \begin{tikzpicture}[scale=0.2, baseline=(current bounding box.center), #1]
    \foreach \x/\y in {#4}
      \fill[pattern=north east lines, pattern color=black!45] (\x,\y) rectangle +(1,1);
    \draw (0.01,0.01) grid (#2+0.99,#2+0.99);
    \foreach \x/\y in {#3}
      \filldraw (\x,\y) circle (5pt);
  \end{tikzpicture}}\;
}

\newcommand{\perm}[1]{\sigma_{#1}}
\newcommand{\symm}{\mathfrak{S}}
\newcommand{\length}{\ell_S}
\newcommand{\reflength}{\ell_T}
\newcommand{\variance}{V}
\DeclareMathOperator{\depth}{dp}
\DeclareMathOperator{\inversions}{Inv}
\DeclareMathOperator{\displacement}{dis}

\newcommand{\Fundamental}[1]{\Phi(#1)}
\newcommand{\SF}{\phi}
\newcommand{\SPF}{\phi'}

\title{Pattern-functions, statistics, and shallow permutations}
\author{Yosef Berman}
\address{Department of Mathematics, Graduate Center (CUNY), New York, NY, USA}
\email{berman.yosef@gmail.com}

\author[Bridget Eileen Tenner]{Bridget Eileen Tenner$^*$}
\address{Department of Mathematical Sciences, DePaul University, Chicago, IL, USA}
\email{bridget@math.depaul.edu}
\thanks{$^*$Research partially supported by NSF Grant DMS-2054436 and Simons Foundation Collaboration Grant for Mathematicians 277603.}

\keywords{}%

\makeatletter
\@namedef{subjclassname@2020}{%
  \textup{2020} Mathematics Subject Classification}
\makeatother
\subjclass[2020]{Primary: 05A05; 
  Secondary: 20F55, 
  05A15
}

\begin{document}
\begin{abstract}
We study relationships between permutation statistics and pattern-functions, counting the number of times particular patterns occur in a permutation. This allows us to write several familiar statistics as linear combinations of pattern counts, both in terms of a permutation and in terms of its image under the fundamental bijection. We use these enumerations to resolve the question of characterizing so-called ``shallow'' permutations, whose depth (equivalently, disarray/displacement) is minimal with respect to length and reflection length. We present this characterization in several ways, including vincular patterns, mesh patterns, and a new object that we call ``arrow patterns.'' Furthermore, we specialize to characterizing and enumerating shallow involutions and shallow cycles, encountering the Motzkin and large Schr\"oder numbers, respectively.
\end{abstract}
\maketitle

\section{Introduction}

Much of the research about permutation patterns has focused on understanding
(characterizing, enumerating, and so on) the permutations that avoid given sets
of patterns. In this work, we take the study of permutation patterns in a
different direction, enumerating the occurrences of a pattern in a permutation
(that is, not just classifying zero/nonzero), and we demonstrate the utility of
this refinement of the standard pattern containment question. 
The general idea of this perspective has also been taken in \cite{bona 1997, bona 1998, burstein, gaetz gao, gaetz ryba, noonan, noonan zeilberger, robertson, robertson wilf zeilberger}, but our approach is rather different from the majority of those works. 
In particular, we explore relationships between certain permutation statistics
and the number of times particular patterns occur in a permutation.  This
appreciation for the quantity of pattern occurrences extends recent work
\cite{ babson steingrimsson, branden claesson, tenner repetition, tenner range-rep}
and suggests that further analyses in this vein may be similarly
fruitful.  

\begin{definition}
Following terminology from \cite{babson steingrimsson}, a
\emph{pattern-function} is a function that can be written as a linear
combination of pattern counts. A pattern-function whose patterns contain no
more than $d$ symbols is a \emph{$d$-function}.
\end{definition}

We demonstrate the value of pattern-functions by using them to calculate the 
statistics studied by Diaconis and Graham in \cite{diaconis graham}.  This will
involve the permutation and, in some cases, a detour via the fundamental bijection.
After establishing this, we use those pattern-functions to answer an open question of
Petersen and the second author about so-called ``shallow'' permutations
\cite{petersen tenner} (equivalently, \cite[Question 3]{diaconis graham}), and we go on to examine special cases of that result more closely.

We will use this type of pattern counting to calculate established permutation
statistics, namely by writing those statistics as linear combinations of certain pattern counts.
Our attention in this paper is focused on 
the four permutation statistics discussed by Diaconis and
Graham \cite{diaconis graham}: length
$$\length(\sigma),$$
reflection length 
$$\reflength(\sigma),$$
\emph{Spearman's disarray} (also called \emph{total displacement} by Knuth \cite{knuth})
\begin{equation}\label{eqn:displacement}
\displacement(\sigma) = \sum_{i = 1}^n|\sigma_i - i|,
\end{equation}
and what we call the \emph{variance} of a permutation
\begin{equation}\label{eqn:variance}
V(\sigma) = \sum_{i = 1}^n (\sigma_i - i)^2.
\end{equation}
We will write each of these four
statistics as linear combinations of pattern counts.

The disarray/displacement statistic $\displacement$ from
Equation~\eqref{eqn:displacement} is related to the \emph{depth} of permutation,
studied by Petersen and the second author \cite{petersen tenner}:
$$\depth(\sigma) = \sum_{\sigma_i > i} (\sigma_i - i).$$
In particular, $\displacement(\sigma) = 2\depth(\sigma)$.
In  \cite{petersen tenner}, they show that 
$$\depth(\sigma) \geq \frac{\length(\sigma) + \reflength(\sigma)}{2}$$ for all $\sigma$, 
and they ask for which permutations this is an equality; i.e., which permutations have minimal depth relative
to their length and reflection length. Note that this is equivalent to \cite[Question 3]{diaconis graham}.

\begin{definition}
A permutation $\sigma$ is \emph{shallow} if its depth is equal to $(\length(\sigma) + \reflength(\sigma))/2$.
\end{definition}

We will answer that question generally using a new type of pattern containment
that we call \emph{arrow patterns}, which in our particular case can also be phrased
in terms of mesh patterns.  In doing so, we give a bijection between shallow
involutions and circles with non-intersecting chords (counted by the Motzkin
numbers), and a bijection between shallow cycles
and separable permutations (counted by the large Schr\"{o}der numbers).
We also prove that a permutation is shallow if and only if its image under the 
fundamental bijection avoids the vincular patterns
$\VinPat{5,24,13}$, $\VinPat{4,25,13}$, and $\VinPat{31,42}$ (Theorem~\ref{thm:shallow criteria}).

The paper is organized as follows.
In Section~\ref{sec:definitions and context}, we introduce terminology and
notation that we will use  throughout the work. We will also make a few simple
observations that recast the familiar permutation statistics \emph{descents}
and \emph{length} in terms of pattern counts, as motivation of the work to come.
Section~\ref{sec:enumerations via pattern counts} uses pattern-functions
to calculate the reflection length, variance, and disarray/displacement of a permutation.
These are addressed in Corollary~\ref{cor:reflection length pattern-function},
Theorems~\ref{thm:variance as four patterns}, and Theorem~\ref{thm:depth
pattern} respectively.  The variance calculation is the first to
rely on the ingenuity of this pattern-function approach. 
Reflection length and disarray/displacement are calculated via pattern-functions
applied to the permutation's image under the fundamental bijection, which will motivate the work
in Section~\ref{sec:characterizing shallow permutations}.
There, we introduce arrow patterns and construct a
pattern-function for $\length$ applied to a permutation's image under the fundamental
bijection. We use this to characterize shallow permutations in
Theorem~\ref{thm:shallow criteria}.
We then apply this in Section~\ref{sec:shallow cycles and involutions} 
to characterize and count shallow involutions and cycles, now without 
requiring arrow patterns, in Corollaries~\ref{cor:motzkin involutions} and~\ref{cor:count shallow cycles}, respectively. 
We conclude with suggestions for further research.

\section{Definitions and context}\label{sec:definitions and context}

Let $\symm_n$ denote the set of permutations of $[1,n] := \{1, \ldots, n\}$. 
A permutation $\sigma \in \symm_n$ can be described in many ways, and we will be using two of these
options. The first, \emph{one-line notation}, writes $\sigma = \sigma_1 \sigma_2
\cdots \sigma_n$ as a word where $\sigma_i:= \sigma(i)$ is the image of $i$
under the map $\sigma$.  The second, \emph{cycle notation}, writes $\sigma$ as a
product of disjoint cycles, each of which describes an orbit of $\sigma$ acting
on $[1,n]$.  For example, one cycle can be written as $(\sigma^t(1) \
\sigma^{t+1}(1) \ \cdots \ \sigma^{-1}(1) \ 1 \ \sigma(1) \ \cdots \
\sigma^{t-1}(1))$ for any $t$.  One-line notation
is the framework for studying permutation patterns,
which we will define below, and
cycle notation will be relevant because our work will make extensive use of the
so-called \emph{fundamental bijection} between permutations written in one-line
notation and permutations written in cycle notation.

\begin{example}\label{ex:421365}
The permutation $421365 \in \symm_6$, written in one-line notation, is written
in cycle notation as $(143)(2)(56)$.  There are
$3 \cdot 1 \cdot 2 \cdot 3! = 36$ ways to write this permutation in cycle
notation, including $(2)(431)(65)$ and $(56)(314)(2)$, and so on.
\end{example}

The fundamental bijection (see \cite{ec1}) starts with a permutation $\sigma$
written in a ``standard representation'' of cycle notation, and maps this
bijectively to a permutation $\Fundamental{\sigma}$ in one-line notation.  The
standard representation of a permutation $\sigma$ writes each cycle with its largest element
in the leftmost position, and writes the cycles from left to right in increasing order of those
largest elements.  The standard representation of the permutation in
Example~\ref{ex:421365} is $(2)(431)(65)$.  The permutation
$\Fundamental{\sigma}$ is obtained by erasing the cycle demarcations, so $\Fundamental{421365} = 243165$.
This process is invertible by noting that every
cycle in $\sigma$ starts with a left-to-right maximum
of $\Fundamental{\sigma}$.
We will sometimes abuse terminology and refer to $\sigma$'s image under
the fundamental bijection as ``the fundamental bijection'' of $\sigma$.

Two sequences $s$ and $t$ over ordered sets are \emph{isomorphic} if, for all
$i$ and $j$, we have $s_i \leq s_j$ if and only if $t_i \leq t_j$.
We write $s \sim t$ to denote that $s$ is isomorphic to $t$.

\begin{definition}
A permutation $\sigma \in \symm_n$ \emph{contains} a \emph{$\pi$-pattern} if a
subsequence of the one-line notation of $\sigma$ is isomorphic to the one-line
notation of $\pi$. Such a subsequence is an \emph{occurrence} of the pattern
pattern $\pi$. Otherwise $\sigma$ \emph{avoids} the pattern $\pi$.
\end{definition}

\begin{example}\label{ex:pattern containment}
The permutation $421365 \in \symm_6$ contains four occurrences of the pattern
$123$: $236, 235, 136, 135$. The permutation avoids the pattern $1234$.
\end{example}

The relationship between a permutation and a pattern is typically studied as a binary question:
contain or avoid? To recognize the value of a more granular analysis, we
introduce the following terminology and notation. The latter is meant to mimic
how coefficients are extracted from polynomials via, for example, $[x^3]f(x)$.

\begin{definition}\label{defn:count}
The \emph{count} of a pattern $\pi$ in $\sigma$ is the number of occurrences of
$\pi$ in $\sigma$. This is denoted $\PatternCount{\pi}{\sigma}$.
\end{definition}

In other words, the count of $\pi$ in $\sigma$ is positive if and only if
$\sigma$ contains a $\pi$-pattern. Example~\ref{ex:pattern containment}
showed that $\PatternCount{123}{421365} = 4$ and
$\PatternCount{1234}{421365} = 0$.

Many specializations of pattern containment exist in the literature, including
vincular, bivincular, barred, and mesh (for references to these and broader
pattern-related topics, see \cite{kitaev}). Of these, the ones primarily
used in this work are vincular patterns, developed in \cite{babson steingrimsson}, in
which portions of the pattern might be required to be bonded together.
A necessary kludge to extend vincular patterns, called arrow patterns 
will be introduced in Section~\ref{sec:characterizing shallow permutations},
but we note at the end of that section that the two
particular arrow patterns we employ can be rewritten as mesh patterns.

\begin{definition}
A \emph{vincular} pattern is a permutation, in which consecutive symbols may be
bonded together.  A permutation $\sigma \in \symm_n$ contains a vincular pattern
if a subsequence of the one-line notation of $\sigma$ occurs in the appropriate
relative order, with bonded numbers appearing consecutively in $\sigma$.  The
bonds will be indicated with underbrackets; that is, if a segment $\alpha$ of
the pattern is required to appear consecutively, then we will write
$\VinPat{\alpha}$.
\end{definition}

We will typically include an underbracket on single numbers that are not bonded
to their neighbors. We do this for aesthetic reasons, and one could certainly
omit it if desired.

\begin{example}
Consider the permutation $421365$.  The vincular pattern $\VinPat{12,3}$ occurs
twice, as $136$ and $135$, and there is one occurrence of $\VinPat{123}$, as $136$.
\end{example}

When discussing an occurrence of a vincular pattern,
we will often indicate the required bonds as a sort of reminder.
So in the permutation $421365$, we could reference the occurrence $\VinPat{13,6} \sim
\VinPat{12,3}$. 
The idea of \emph{count} from Definition~\ref{defn:count} carries over to the
setting of vincular patterns.
For example, $\PatternCount{\VinPat{12,3}}{421365} = 2$ and $\PatternCount{\VinPat{123}}{421365} = 1$.

We can now make our first characterization (and a very easy one, at that) of a
classical permutation statistic in terms of pattern counts.

\begin{observation}\label{obs:descents as patterns}
For any permutation $\sigma$, the number of descents in $\sigma$ is equal
to $\VinCount{21}(\sigma)$.
\end{observation}

Our goal is to use pattern counts to enumerate a range of phenomena.  The
notation that we employ for this is similar to that introduced in
Definition~\ref{defn:count}, treating pattern counts as operators
and summing them as needed.
A simple example of this is that for any $\sigma \in \symm_n$,
values appearing in consecutive positions either increase or decrease.
There are $n-1$ pairs of consecutive positions, and so we have the pattern-function
$$\VinCount{12}(\sigma) + \VinCount{21}(\sigma) = n-1.$$
Because this is true for all $\sigma \in \symm_n$, we will write the general fact as
$$\left(\VinCount{12} + \VinCount{21}\right)(\sigma) = n-1,$$
or, when the meaning of $\sigma$ is clear, simply as
\begin{equation}\label{equation: sum of consecutive}
\VinCount{12} + \VinCount{21} = n-1.
\end{equation}

The set $\symm_n$ is a Coxeter group, and the length and reflection length of a
permutation give a notion of the permutation's complexity as a Coxeter group
element.  More precisely, let $S$ be the set of simple reflections (adjacent
transpositions) and let $T$ be the set of all reflections (conjugates of the
simple reflections).  Then the \emph{length} of a permutation $\sigma$, denoted
$\length(\sigma)$, is the minimal number of elements of $S$ needed to form a
product equalling $\sigma$ and the \emph{reflection length} of $\sigma$, denoted
$\reflength(\sigma)$, is the minimal number of elements of $T$ needed to form a
product equalling $\sigma$. That is,
\begin{align*}
\length(\sigma) &= \min\{k : \sigma = s_1 \cdots s_k \text{ for } s_i \in S\} \text{ and}\\
\reflength(\sigma) &= \min\{k : \sigma = t_1 \cdots t_k \text{ for } t_i \in T\}.
\end{align*}

It is well known that the number of inversions equals a permutation's length, which we can now write as follows.

\begin{observation}\label{obs:length as patterns}
$\length(\sigma) = \VinCount{2,1}(\sigma)$
\end{observation}

\section{Enumerations via pattern counts}\label{sec:enumerations via pattern counts}

Diaconis and Graham consider four statistics: $\length(\sigma)$,
$\reflength(\sigma)$, $\displacement(\sigma)$, and the variance $V(\sigma)$.
Observation~\ref{obs:length as patterns}
translated the first of these into the language of patterns, and our goal is to
do similarly for the other metrics.

Reflection length, like length, can be written as a pattern-function. Unfortunately, its
formulation is less clean, requiring a sum whose terms depends on $n$.
We note at the end of the section why reflection length is not
a $d$-function for any finite $d$.
To start, we use a result of Br\"and\'en and Claesson
to write the number of cycles in a permutation in terms of pattern counts in its
image under the fundamental bijection. 

\begin{proposition}[{cf.~\cite[Proposition 3]{branden claesson}}]\label{prop:branden claesson}
The number of left-to-right maxima in a permutation $\sigma$ is
$$\sum_{k \ge 1} (-1)^{k-1} \sum_{\substack{\pi \in \symm_k\\ \pi(k) = 1}}\PatternCount{\pi}{\sigma}.$$ 
\end{proposition}

Because the sum of a permutation's reflection length and number of cycles is equal to its size, Proposition~\ref{prop:branden claesson} allows reflection length to be written in terms of pattern counts in the fundamental bijection. 

\begin{corollary}\label{cor:reflection length pattern-function}
For $\sigma \in \symm_n$,
let $\Fundamental{\sigma}$ be the fundamental bijection of $\sigma$.
Then $$\reflength(\sigma) = n - \sum_{k \ge 1} (-1)^{k-1} \sum_{\substack{\pi \in \symm_k\\ \pi(k) = 1}}\PatternCount{\pi}{\Fundamental{\sigma}}.$$
\end{corollary}

Compared to length and reflection length, calculating a pattern-function
for variance is a much more interesting challenge. To do so, we first characterize variance in terms of the inversion set of a permutation.
Our proof uses some handy properties of summations, but the result can also be
proved inductively or by a geometric argument. 

\begin{lemma}\label{lemma:sumOfSquaresIsInversionDiffs}
For any permutation $\sigma \in \symm_n$, 
$$V(\sigma) = 2\hspace{-.1in}\sum_{(i,j) \in \inversions(\sigma)} (\sigma_i - \sigma_j).$$
\end{lemma}

\begin{proof}
Because $\sigma$ is a permutation, we can write
$$V(\sigma) = \sum_{i = 1}^{n} (\sigma_i - i)^2 = \sum_{i=1}^n\left(\sigma_i^2 - 2i\sigma_i + i^2\right) = 2\sum_{i=1}^ni^2 - 2\sum_{i=1}^ni\sigma_i.$$
To simplify the right-hand side of the proposition, we note that
$$x-y + |x-y| = \begin{cases} 
2(x-y) & \text{if } x > y, \text{ and} \\ 
0 & \text{otherwise.} 
\end{cases}$$
Thus
\begin{align}
2\hspace{-.1in}\sum_{(i,j) \in \inversions(\sigma)} (\sigma_{i} - \sigma_{j}) &= \sum_{i<j} \left(\sigma_i - \sigma_j + |\sigma_i-\sigma_j|\right)\nonumber\\
&= \sum_{i<j} \sigma_i - \sum_{i<j} \sigma_j + \sum_{i<j} |\sigma_i-\sigma_j|\nonumber\\
&= \sum_{i=1}^n \sigma_i(n-i) - \sum_{i=1}^n \sigma_i(i-1) + \sum_{i<j} |\sigma_i-\sigma_j|\nonumber\\
&= (n+1)\sum_{i=1}^n i - 2\sum_{i = 1}^n i\sigma_i + \sum_{i<j} |\sigma_i-\sigma_j|.\label{eqn:expanding inversions}
\end{align}

The sets $\{\{i, j\} : 1 \leq i < j \leq n \}$ and
$\{\{\sigma_i, \sigma_j\} : 1 \leq i < j \leq n \}$
are equal, and we can partition the $\binom{n}{2}$ pairs $\{x,y\}$ in either of them by the difference $|x-y|$. Therefore
\begin{equation}\label{eqn:expanding inversions part 2}
\sum_{i<j} |\sigma_i-\sigma_j| = \sum_{i<j}(j-i) = \sum_{\delta = 1}^n \delta(n-\delta) = n\sum_{\delta=1}^n\delta - \sum_{\delta=1}^n\delta^2.
\end{equation}
Combining Equations~\eqref{eqn:expanding inversions} and~\eqref{eqn:expanding inversions part 2} yields
$$2\hspace{-.1in}\sum_{(i,j) \in \inversions(\sigma)} (\sigma_i - \sigma_j) = (2n+1)\sum_{i=1}^n i - 2\sum_{i=1}^ni\sigma_i - \sum_{i=1}^ni^2.$$
Summation identities for $\sum i$ and $\sum i^2$ complete the proof.
\end{proof}

We are now ready to write the variance $V(\sigma)$ as a $3$-function.

\begin{theorem}\label{thm:variance as four patterns}
$V(\sigma) = 2(\Count{21} + \Count{231} + \Count{312} + \Count{321})(\sigma)$.
\end{theorem}

\begin{proof}
Lemma~\ref{lemma:sumOfSquaresIsInversionDiffs} gave
$$V(\sigma) = 2\hspace{-.1in}\sum_{(i,j) \in \inversions(\sigma)} (\sigma_i - \sigma_j),$$
and we will show that this sum is equal to
$$(\Count{21} + \Count{231} + \Count{312} + \Count{321})(\sigma).$$
Inversions are counted by $\Count{21}$. For each inversion $(i,j)$,
there are $\sigma_i-\sigma_j-1$ numbers $a$ such that
$\sigma_j < a < \sigma_i$.
Each of those numbers is counted exactly once by the sum $(\Count{21} + \Count{231} + \Count{312} + \Count{321})(\sigma)$. More specifically, the position of $2$ relative to $3$ and $1$ in the pattern corresponds to the value $\sigma^{-1}(a)$ relative to $i$ and $j$.
\end{proof}

Similar to reflection length (and unlike length and variance), our
pattern-function for the disarray/displacement of a permutation is in terms of the
permutation's image under the fundamental bijection.

\begin{theorem}\label{thm:depth pattern} 
For any permutation $\sigma$, we have
$\displacement(\sigma) = 2\left(\VinCount{21} + \VinCount{2,31} + \VinCount{31,2}\right)(\Fundamental{\sigma})$.
\end{theorem}

\begin{proof}
For ease of notation, set $\phi := \Fundamental{\sigma}$. We will actually prove an equivalent statement, that the depth
$\depth(\sigma)$ is equal to $$\left(\VinCount{21} + \VinCount{2,31} +
\VinCount{31,2}\right)(\phi).$$
The fundamental bijection guarantees that if
$\phi_{j} > \phi_{j+1}$
then $\phi_{j}$ and $\phi_{j+1}$ are in the same cycle, and
$\sigma_{\phi_{j}} = \phi_{j+1}$.

Consider the standard cycle representation of $\sigma$.  If $i > \sigma_{i}$,
then $\sigma_{i}$ is not the largest element in its cycle and hence
$\sigma_{i}$ is not written first in the cycle in standard form.  Let $j$ be
such that $i = \phi_{j-1}$, and hence $\sigma_{i} = \phi_{j}$. Then we can
write
\begin{equation}
\label{eqn:understanding depth} i - \sigma_{i} = \phi_{j-1}- \phi_{j} = 1 + \#\{x :  \phi_{j}< x < \phi_{j-1}\}.
\end{equation}
Because
$\sum_i (\sigma_{i} - i) = \sum_{i \neq \sigma_{i}} (\sigma_{i}-i) = 0$,
we have
$$\depth(\sigma) = \sum_{\sigma_{i} > i}(\sigma_{i}-i) = -\sum_{i > \sigma_{i}}(\sigma_{i}-i) = \sum_{i > \sigma_{i}}(i-\sigma_{i}).$$
Consider the
rightmost expression in Equation~\eqref{eqn:understanding depth}.  The $1$ in
that sum identifies the $\VinPat{21}$-pattern formed by
$\phi_{j-1}\phi_j = i\perm{i}$ in $\phi$.  Now consider the locations of all
$\{x :  \phi_{j}< x < \phi_{j-1}\}$ in the one-line notation for $\phi$. If
$x$ appears to the left of $\phi_{j-1}$, then the values
$\{x,\phi_{j-1},\phi_j\}$ form a $\VinPat{2,31}$-pattern in $\phi$.
Otherwise those letters form a $\VinPat{31,2}$-pattern in $\phi$. Thus
$$\sum_{i>\perm{i}} (i-\sigma_i) = \left(\VinCount{21} + \VinCount{2,31} + \VinCount{31,2}\right)(\phi).$$
\end{proof}

We close this section with an application of these formulas. 
Using \cite[Proposition 4]{babson steingrimsson}
and the linearity of expectation, the pattern-function formulation
of each of these statistics allow us to systematically determine their expected value.
For example, Observation~\ref{obs:length as patterns} gives us
$E\left[\length\right] =\frac{n^2 - n}{4}$,
Theorem~\ref{thm:variance as four patterns} gives us
$E\left[\variance\right] = \frac{n^3-n}{6}$,
and Theorem~\ref{thm:depth pattern} gives us
$E\left[\displacement\right] = 2E\left[\depth\right] = \frac{n^2 - 1}{3}$.
We could similarly find a polynomial in $n$ for the expected value of any finite pattern-function.
A more exciting example, from Corollary~\ref{cor:reflection length pattern-function}, first requires a summation formula.

\begin{proposition}[{\cite[\S6.4]{graham knuth patashnik}}]
For $n > 0$, we have
$$\sum_{k = 1}^{n}(-1)^{k-1}\binom{n}{k}\frac{1}{k} = \sum_{k = 1}^{n}\frac{1}{k},$$
the $n$th harmonic number $H_n$.
\end{proposition}

From this we see $E\left[\reflength\right] = n - H_n$ which is equivalent to the
classic result that the expected number of cycles in a uniformly selected
permutation is the $n$th harmonic number.
Furthermore, there cannot be a
finite length pattern-function for either reflection length or the number of cycles, since
their expected value is not a polynomial in $n$.

\section{Characterizing shallow permutations}\label{sec:characterizing shallow permutations}

In order to address the question of \cite{petersen tenner} to characterize
shallow permutations, it is necessary to consider length, reflection length,
and depth simultaneously. Thus we want the pattern-functions for those
statistics to be comparable; namely, all in terms of the same object. While
Observation~\ref{obs:length as patterns} computes length as a simple
pattern-function of the permutation,
Corollary~\ref{cor:reflection length pattern-function}
writes reflection length as a pattern-function of the
fundamental bijection, as does Theorem~\ref{thm:depth pattern} for depth. Thus
it is necessary to recalculate length in terms of pattern counts in the
fundamental bijection, which we do in Theorem~\ref{thm:length of perm}.

The expected values of length and depth are polynomial in $n$. The expected
value of reflection length is not, so neither is the expected value of
$\depth(\sigma)-(\length(\sigma)+\reflength(\sigma))/2$.
We conclude that there is no finite pattern-function for this difference
using only vincular patterns.  Thus we introduce a new type of pattern that simultaneously
captures some information from $\sigma$ and some from $\Fundamental{\sigma}$.
This is a generalization of vincular patterns directly applied to
$\Fundamental{\sigma}$, which we call \emph{arrow patterns}.
At the end of the section we briefly mention how to rewrite the arrow patterns we
use in this result as a linear combination of mesh patterns.

\begin{definition}
An \emph{arrow pattern} $\alpha$ in $\symm_k$ consists of
\begin{itemize}
\item sets of integers $A = \{a_1, \ldots, a_m\}$, $B = \{b_1, \ldots, b_h\}$, and $C = \{c_1, \ldots, c_h\}$ where $A \cup B \cup C = [1, k]$,
\item a string $\nu = a_1 \cdots a_m$, in which some consecutive symbols may be bonded together (this resembles a vincular pattern, although $A$ may not equal $[1,m]$), and
\item a (possibly empty) collection of $h$ arrows $\{b_i \rightarrow c_i : i = 1, \ldots, h\}$.
\end{itemize}
Let $\tau \in \symm_r$ be a permutation, and define $\sigma$ so that $\Fundamental{\sigma} = \tau$. Let $\omega := x_{a_1}x_{a_2} \cdots x_{a_m} = \tau_{t_1} \cdots \tau_{t_m}$ be a substring of the one-line notation for $\tau$. The subsequence $\omega$ is an \emph{occurrence} of the arrow pattern $\alpha$ in $\tau$ if
\begin{itemize}
\item the substring $\omega$ is order isomorphic to $\nu$,
\item if $a_j$ and $a_{j+1}$ are bonded in $\nu$, then $x_{a_{j}}$ and $x_{a_{j+1}}$ are adjacent in $\tau$ (that is, $t_{j+1} = t_j+1$), and
\item there exists $X = \{x_1 < \cdots < x_k\} \subseteq [1,r]$ such that $\{x_{a_1}, x_{a_2}, \ldots, x_{a_m}\}\subseteq X$ and, for every arrow $b_i \rightarrow c_i$, we have $\sigma(x_{b_i}) = x_{c_i}$ for $1\leq i \leq h$.
\end{itemize}
\end{definition}

Note that an \emph{occurrence} of an arrow pattern refers to
an occurrence of the underlying vincular pattern, subject to the additional constraints imposed by arrows. 
Generally speaking, arrow patterns are ripe for further study. In particular, they could be poised to bridge the divide between cycle notation and one-line notation, enabling the application of pattern techniques and results to a substantially broader range of questions. 
 For this paper we will be concerned with patterns which
contain a single arrow. We will also always have one end
of the arrow contained in the vincular portion.

\begin{example}
The permutation $\tau = 63248175$ has $\sigma = \Phi^{-1}(\tau) = 74268351$.
The permutation $\tau$ has two occurrences of the arrow pattern $\AVinP{12}{1/2}$; namely, $24$ and $17$.
The occurrence $24$ matches $\VinPat{12}$ because $\sigma(2) = 4$ letting $X = \{2 < 4\}$.
Similarly, $17$ matches $\VinPat{12}$ where $X=\{1< 7\}$ and $\sigma(1) = 7$.
Although $48$ matches $\VinPat{12}$, it is not an occurrence of the given arrow pattern because $\sigma(4) = 6 \neq 8$.
  The subsequence $48$ does, however, match $\AVinP{13}{1/2}$ where $X = \{4 < 6 < 8\}$ and $\sigma(4) = 6$.
\end{example}

Arrow patterns offer many pattern coincidences stemming from the presence of the fundamental bijection in the definition, and we highlight some identities
that we will use subsequently.  Sometimes, arrow patterns might simplify to a
vincular pattern, due to properties of the fundamental bijection.  For
example, with a descent $ab$, we would not
need to write $a \shortrightarrow b$, so we have coincidences like
$$\AVinP{21}{2/1} = \VinPat{21} \text{ \ and \ } \AVinP{2,43}{2/1} = \VinPat{21,43}.$$
Sometimes, the arrow expressions imply bonds, as with
$$\AVinP{1,2}{1/2} = \AVinP{12}{1/2},$$
but we may need to keep the arrow itself.
For example, the identity permutation, which is its own image under the fundamental bijection, contains $\VinPat{12}$ but avoids the arrow patterns listed above. 
Some equalities are subtle, such as
\begin{equation}\label{eqn:{1,3}{1/2} = {2,3}{1/2}}
\AVinC{1,3}{1/2} = \AVinC{2,3}{1/2} \text{ \ and \ } \AVinC{1,43}{1/2} =
\AVinC{2,43}{1/2},
\end{equation}
which result from knowing that the $1$ and $2$ come before
the $\VinPat{3}$ and $\VinPat{43}$ respectively.  Note in these last cases that
the pattern counts are equal, but the patterns themselves are not.

We now present a finite pattern-function for reflection length using arrow patterns.

\begin{proposition}\label{prop: reflength as pairs} 
$\reflength(\sigma) = \left(\VinCount{21} + \AVinC{12}{1/2}\right)(\Fundamental{\sigma}).$
\end{proposition}

\begin{proof}
Set $\phi:=\Fundamental{\sigma}$, and suppose that $\sigma \in \symm_n$ has $c$ cycles. It is well-known that $\reflength(\sigma) = n-c$. There are $c-1$ indices $i$ such that $\phi_i$ and $\phi_{i+1}$ are in distinct cycles.
Thus there are $(n-1)-(c-1) = \reflength(\sigma)$ indices $i$ such that $\sigma(\phi_i) = \phi_{i+1}$.
Assume $\sigma(\phi_i) = \phi_{i+1}$ and
let $x := \phi_i$ and $y := \phi_{i+1}$.
There are two cases:
either $x > y$ in which case $\VinPat{xy} \sim \VinPat{21}$,
or $x < y$ and so $\VinPat{xy} \sim \AVinP{12}{1/2}$.
\end{proof}

The pattern-function for depth that arises from Theorem~\ref{thm:depth pattern} is not ideal for comparing with the pattern-function for length that we will develop
below in Theorem~\ref{thm:length of perm}. To address this, we introduce an alternative pattern-function for depth, using arrow patterns.

\begin{theorem}\label{thm:depth of perm}
$\depth(\sigma) = \reflength(\sigma) + \left(\VinCount{2,31} + \VinCount{41,32} +
\VinCount{31,42} + \AVinC{1,23}{1/4} + \AVinC{2,13}{2/4}\right)(\Fundamental{\sigma}).$
\end{theorem}

\begin{proof}
Throughout this proof, let all vincular pattern counts be
applied to the fundamental bijection $\Fundamental{\sigma}$.
Given a $\VinPat{x,y} \sim \VinPat{1,2}$ occurrence in $\Fundamental{\sigma}$
there are five possible relative values for
$\sigma(x)$ and four possible relative values for $\sigma^{-1}(x)$. Taken together, these yield
$$\VinCount{1,2}
  = \VinCount{21,3} + \AVinC{1,2}{1/1} + \AVinC{1,3}{1/2} + \AVinC{1,2}{1/3} + \AVinC{12}{1/2} = \AVinC{2,3}{1/2} + \AVinC{1,2}{1/1} + \VinCount{21,3} + \VinCount{31,2}.$$
By subtracting common terms and using Equation~\eqref{eqn:{1,3}{1/2} = {2,3}{1/2}}, this reduces to
$\AVinC{1,2}{1/3} + \AVinC{12}{1/2} = \VinCount{31,2}$.

Proposition~\ref{prop: reflength as pairs} yields
\begin{equation}\label{another useful equation}
\VinCount{31,2} = \reflength(\sigma) + \AVinC{1,2}{1/3} - \VinCount{21}.
\end{equation}
Similarly, we can analyze $\VinCount{1,32}$ to get
$$\VinCount{21,43} + \AVinC{1,32}{1/1} + \AVinC{1,43}{1/2} + \AVinC{1,42}{1/3} + \AVinC{1,32}{1/4} + \AVinC{132}{1/3}
= \AVinC{2,43}{1/2} + \AVinC{1,32}{1/1} + \VinCount{21,43} + \VinCount{31,42} + \VinCount{41,32},$$
and Equation~\eqref{eqn:{1,3}{1/2} = {2,3}{1/2}} allows us to reduce this to
\begin{equation}\label{useful equation}
\AVinC{1,42}{1/3} + \AVinC{1,32}{1/4} + \AVinC{132}{1/3} = \VinCount{31,42} + \VinCount{41,32}.
\end{equation}
Considering the element prior to the ``$2$'' in the arrow pattern $\AVinC{1,2}{1/3}$ produces the identity
$$\AVinC{1,2}{1/3}
= \AVinC{2,13}{2/4} + \AVinC{1,23}{1/4} + \AVinC{1,32}{1/4} + \AVinC{1,42}{1/3} + \AVinC{132}{1/3}$$
which combines with Equation~\eqref{useful equation} to yield
$$\AVinC{1,2}{1/3} =
  \VinCount{31,42} + \VinCount{41,32} + \AVinC{2,13}{2/4} + \AVinC{1,23}{1/4}$$
The result now follows from Equation~\eqref{another useful equation} and Proposition~\ref{thm:depth pattern}.
\end{proof}

We can now calculate the length of an arbitrary permutation by enumerating arrow pattern occurrences in its image under the fundamental bijection. The benefit of this, as opposed to the straightforward statement in Observation~\ref{obs:length as patterns}, is that it can be easily compared with Theorem~\ref{thm:depth of perm}, and hence will enable us to address shallow permutations.

\begin{theorem}\label{thm:length of perm}
Let $\sigma \in \symm_n$ be an arbitrary permutation. Then 
\begin{equation}
\label{eqn: length formula}\length(\sigma) =
\reflength(\sigma) + 2\left(\VinCount{2,31} +
\VinCount{41,32} + \AVinC{1,23}{1/4}\right)(\Fundamental{\sigma})
\end{equation}
\end{theorem}

\begin{proof}
We prove the claim by inducting on $\reflength(\sigma)$.
If $\reflength(\sigma) = 0$ then $\sigma = \identity$, the identity permutation.
The theorem holds in this case because $\Fundamental{\identity} = \identity$ and all quantities in the statement of the theorem would be $0$.

Assume Equation~\eqref{eqn: length formula} holds for all permutations with reflection length less than some $k>0$, and we will prove that it also holds when $\reflength(\sigma) = k$.
We start by defining:
\begin{align*}
 h &:= \min \{ x : \sigma_x \neq x\},\\
 i &:= (\sigma^{-1})_h,\\
 r &:= \sigma_h, \text{ and}\\
\sigma' &:= (r\ h)\sigma.
\end{align*}
Note that minimality of $h$ means $h < i, r$. The one-line notations of $\sigma$ and $\sigma'$ differ only in the positions of $r$ and $h$, with
$$\sigma = 123(h-1)r\cdots h\cdots \text{ \ and \ } \sigma'= 123(h-1)h\cdots r\cdots.$$
In standard cycle representation,
the $h$ is removed from its cycle in $\sigma$ and inserted after $(1)(2)\cdots(h-1)$
as an additional fixed point in $\sigma'$:
$$\sigma = (1)\cdots(h-1)\cdots(\cdots\ i\ h\ r \cdots) \cdots  \text{ \ and \ }
\sigma'= (1)\cdots(h-1)(h) \cdots (\cdots\ i\ r \cdots) \cdots.$$

Let $K := \{k : h < k < i \text{ and } r > \sigma_k > h\}$.
Consider the inversion set $\inversions(\sigma')$ and define
$$\inversions^+(\sigma') := \{(a,j) \in \inversions(\sigma') : a \neq i\} \sqcup \{(h,j) : (i,j) \in \inversions(\sigma')\},$$
where $|\inversions^+(\sigma')| = |\inversions(\sigma')|$ and $\inversions^+(\sigma') \subseteq \inversions(\sigma)$.
The inversion set $\inversions(\sigma)$ can be written as the disjoint union
$$\inversions(\sigma) = \inversions^+(\sigma') \sqcup \{(h, k), (k, i) : k \in K\} \sqcup \{(h, i)\},$$
which implies that
$$\length(\sigma) = \length(\sigma') + 2|K| + 1.$$

Because a permutation's reflection length and number of cycles sum to its size, and because $\sigma'$ has one more cycle than $\sigma$, we have 
$\reflength(\sigma) = \reflength(\sigma') + 1$.
Set
$$P := \{\VinPat{2,31}, \VinPat{41,32}, \AVinP{1,23}{1/4}\},$$
and call an occurrence of any of these three patterns a \emph{$P$-pattern}.
For a permutation $\pi$, we write $P(\pi)$ to denote all $P$-patterns in $\pi$.
For the remainder of this proof, set $\SF := \Fundamental{\sigma}$ and $\SPF := \Fundamental{\sigma'}$. Therefore by the inductive hypothesis it suffices to show that $|K|$ is equal to
\begin{equation}
\label{eqn:difference in pattern counts full}\
|P(\SF)| - |P(\SPF)|.
\end{equation}

Recall the definition of $h$, meaning that the first $h-1$ values in $\SF$ and $\SPF$ are fixed. Thus if $h$ appears in any $P$-pattern in either permutation, then $h$ must be the smallest value in the occurrence.

We will define an injection $\chi: P(\SPF) \to P(\SF)$, with $X := P(\SF) \setminus \chi(P(\SPF))$. We will then construct a bijection $\rho: X \to K$ to complete the proof. The map $\chi$ is defined as follows:

\begin{quote}
\begin{itemize}
\item[Case~0:] 
If $p \in P(\SPF)$ and $p \in P(\SF)$, meaning that $p$ does not use the letter $h$, then $\chi(p) := p$.
Otherwise, either a bonded group or the arrow of $p$ in $\SPF$
gets ``interrupted'' by the $h$ in $\SF$, which can occur in five different ways.

\item[Case~1:]
An occurrence of $\VinPat{2,3 1}$ appears in $\SPF$ but not in $\SF$ if and only if it is $\VinPat{x,i r} \in P(\SPF)$. Thus $\VinPat{x,i h r} \sim \VinPat{3,4 1 2}$ in $\SF$, and we set $\chi(\VinPat{x, i r}) := \VinPat{x, i h}$.

\item[Case~2:]
If $\VinPat{i r,x y} \sim \VinPat{41,32} \in P(\SPF)$ then
$\VinPat{i h r,x y} \sim \VinPat{512,43}$ in $\SF$.
Set $\chi(\VinPat{i r,x y}) := \VinPat{ih,xy}$.

\item[Case~3:]
If $\VinPat{x y,i r} \sim \VinPat{4 1,3 2} \in P(\SPF)$ then
$\VinPat{x y,i h r} \sim \VinPat{5 2,4 1 3}$ in $\SF$.
Set $\chi(\VinPat{x y,i r}) := \VinPat{y, i h} \sim \VinPat{2,31}$.

\item[Case~4:]
If $\VinPat{i, x y} \sim \AVinP{1,2 3}{1/4} \in P(\SPF)$ then
$h$ in $\SF$ interrupts the arrow.
By definition of the fundamental bijection, $\VinPat{ih,xy} \sim \AVinP{2 1,3 4}{1/5}$ in $\SF$.
Set $\chi(\VinPat{i, x y}) := \VinPat{h,x y} \sim \AVinP{1,23}{1/4}$.

\item[Case~5:]
Finally, if $\VinPat{x,i r} \sim \AVinP{1,2 3}{1/4}$ then
$\VinPat{x,i h r} \sim \AVinP{2,3 1 4}{2/5}$ in $\SF$.
Set $\chi(\VinPat{x,i r}) := \VinPat{x, i h} \sim \VinPat{2,31}$.
\end{itemize}
\end{quote}
A straightforward case analysis shows that the map $\chi$ is injective.

Consider an element $p \in X$, meaning that $p$ is a $P$-pattern occurring in $\SF$ that is not in the image of $\chi$. To avoid Case~0, the occurrence $p$ must use $h$, and $h$ must be minimal in the occurrence $p$, as discussed above. 
If $\VinPat{x, i h} \sim \VinPat{2, 3 1} \in X$, then 
to avoid Cases~1, 3, and~5, we must have
\begin{align*}
&r > x,\\
&\text{if } r < i \text{ then } \sigma^{-1}(x) < i, \text{ and}  \\
&\text{if }r \ge i \text{ then } \sigma_x < r.
\end{align*}
If $r < i$, then $h < \sigma^{-1}(x) < i$ and $r > x > h$, so set $\rho(\VinPat{x,ih}) := \sigma^{-1}(x)$.
If $r \ge i$, then $h < x < i$ and $r > \sigma_x > h$, and we set $\rho(\VinPat{x,ih}) := x$.
Every $k \in K$ with the position
of $k$ in $\SF$ coming before the position of $h$ must
fall into either one of these two options.

If $\VinPat{ih,xy} \sim \VinPat{41,32} \in X$, then $\sigma_x = y$ and to
avoid Case~2 we must have $r > y$.  Then $h < x < i$ and $r > y > h$, and we set
$\rho(\VinPat{ih, xy}) := x$.  If $\VinPat{h,xy} \sim \AVinP{1,23}{1/4}$, then
$\sigma_x = y$ and $\sigma_h = r > y$.  To avoid Case~4 we must have $i > x$.
Then $h < x < i$ and $r > y > h$, and we set $\rho(\VinPat{h,xy}) := x$.
Every $k \in K$ with the position
of $k$ in $\SF$ coming after the position of $h$ must
fall into one of these two options.
Inverting $\rho$ is straightforward based on
whether a given $k \in K$ appears before or after $h$.
\end{proof}

Theorems~\ref{thm:depth of perm} and~\ref{thm:length of perm}
recover the inequality \cite[Observation 2.2]{petersen tenner}:
$$\frac{\reflength(\sigma) + \length(\sigma)}{2} \leq \depth(\sigma) \leq \length(\sigma),$$
or, in the language of \cite{diaconis graham},
$\reflength(\sigma) + \length(\sigma) \leq \displacement(\sigma) \leq  2\length(\sigma)$.
Furthermore, they show the following.

\begin{corollary}\label{thm:depth vs length}
For any permutation $\sigma$,
$$\depth(\sigma) = \displaystyle{\frac{\length(\sigma) + \reflength(\sigma)}{2}} + \left(\VinCount{31,42} + \AVinC{2,13}{2/4}\right)(\Fundamental{\sigma}).$$
In particular, a permutation is shallow if and only if its image under the fundamental bijection avoids
$\VinPat{31,42}$ and $\AVinP{2,13}{2/4}$.
\end{corollary}

We now can produce a characterization of shallow permutations that does not depend on arrow
patterns.

\begin{theorem}\label{thm:shallow criteria}
A permutation $\sigma_n \in \symm_n$ is shallow if and only
if its image, $\Fundamental{\sigma}$, under the fundamental bijection avoids
$$\{\VinPat{5,24,13}, \VinPat{4,25,13}, \VinPat{31,42}\}.$$
\end{theorem}

\begin{proof}
From Corollary~\ref{thm:depth vs length}, a permutation
$\sigma \in \symm_n$ is shallow if and only if $\Fundamental{\sigma}$
avoids $$\{\VinPat{31,42}, \AVinP{2,13}{2/4}\}.$$
We will now show that containment (resp., avoidance) of the
latter of these two patterns is equivalent to containment
(resp., avoidance) of the patterns
$\{\VinPat{5,24,13}, \VinPat{4,25,13}\}$.

For $\Fundamental{\sigma}$ to contain
$\VinPat{x,yz} \sim \AVinP{2,13}{2/4}$,
it must be that $$y < x < z < \sigma(x).$$
Either $\sigma(x)$ appears to the right of $x$ in the one-line notation
for $\Fundamental{\sigma}$,
meaning that $\sigma(x)$ is not the largest element
in its cycle in $\sigma$,
or $\sigma(x)$ appears to the left of $x$ in the one-line
notation for $\Fundamental{\sigma}$
because $\sigma(x)$ is the largest element
in its cycle in $\sigma$.
Consider first the former case.
Then the elements $\{x,\sigma(x), y,z\}$ together with
the largest element of $x$'s cycle in $\sigma$ form a
$$\VinPat{5,24,13}\text{-pattern}$$
in $\Fundamental{\sigma}$.
In the latter case, we can make several observations.
First, $y$ (and $z$) is not in the same cycle as
$\{x,\sigma(x)\}$, because $x$'s cycle ends after $x$.
Second, the largest element of another cycle 
appears immediately to the right of $x$ in the one-line
notation of $\Fundamental{\sigma}$, and it must be larger
than $\sigma(x)$,
by definition of the fundamental bijection.
Therefore the elements $\{\sigma(x),x,y,z\}$ together with
the element immediately after $x$ in $\Fundamental{\sigma}$
form a $$\VinPat{4,25,13}\text{-pattern}$$
in $\Fundamental{\sigma}$.

Now suppose that $\Fundamental{\sigma}$ contains
$\VinPat{a,xb,yz} \sim\VinPat{5,24,13}$.
Left-to-right maxima in $\Fundamental{\sigma}$
indicate the beginning of new cycles in $\sigma$,
so we must have $b = \sigma(x)$.
Therefore $\VinPat{x,yz} \sim \AVinP{2,13}{2/4}$.

If, on the other hand, $\Fundamental{\sigma}$ contains
$\VinPat{a,xb,yz} \sim\VinPat{4,25,13}$,
then either $\sigma(x) = b$, in which case we have
$\VinPat{x,yz} \sim \AVinP{2,13}{2/4}$, or $\sigma(x) \neq b$.
Suppose $\sigma(x) \neq b$, so that $\sigma(x)$ is the
largest element in $x$'s cycle in $\sigma$ and it appears
to the left of $x$ in the one-line notation for
$\Fundamental{\sigma}$. If $a$ is in the same cycle as $x$,
then $\sigma(x) \ge a$ and hence $\VinPat{x,yz} \sim \AVinP{2,13}{2/4}$.
Otherwise, $\sigma(x)$ is the largest element in a cycle appearing to the right of the cycle containing $a$
in the standard representation of $\sigma$, meaning once again that $\sigma(x) \ge a$ and so $\VinPat{x,yz} \sim \AVinP{2,13}{2/4}$.

Therefore $\AVinP{2,13}{2/4}$-avoidance is equivalent to $\{\VinPat{5,24,13}, \VinPat{4,25,13}\}$-avoidance, completing the proof.
\end{proof}

The arrow patterns provided above are novel, which might leave the reader
dissatisfied. For those who prefer mesh patterns, we can also use those constructions to express these results.

\begin{proposition}
$$\AVinP{1,23}{1/4} =
\MeshPattern{}{4}{1/1, 2/4, 3/2, 4/3}{1/0, 1/1, 1/2, 1/3, 1/4, 3/0, 3/1, 3/2, 3/3, 3/4}
-\MeshPattern{}{4}{1/1, 2/4, 3/2, 4/3}{0/3, 0/4, 1/0, 1/1, 1/2, 1/3, 1/4, 3/0, 3/1, 3/2, 3/3, 3/4}$$
\end{proposition}

\begin{proof}
Suppose  $\VinPat{x,yz}$ is an occurrence of $\AVinP{1,23}{1/4}$ in $\tau$, with $\Fundamental{\sigma} = \tau$.  Then either
$\sigma(x)$ follows $x$ in the one-line notation of $\tau$, or $\sigma(x)$ is at the beginning of $x$'s cycle in $\sigma$ and thus appears to the left of $x$ in $\tau$.
In either case, let $w$ be the value immediately after $x$ in $\tau$. By definition of the fundamental bijection, $w \ge \sigma(x)$, so we must have $\VinPat{xw,yz} \sim \VinPat{14,23}$.

Now suppose that $\VinPat{xw,yz} \sim \VinPat{14,23}$ but
$\VinPat{x,yz}$ is not an occurrence of $\AVinP{1,23}{1/4}$.
We must have $\sigma(x) < z$, so that 
$\sigma(x) < w$ and $\sigma(x)$ is the largest element of its cycle (appearing to the left of $x$ in $\tau$).
Equivalently, there is no element greater than $z$ to the left of $w$ in the one-line notation for $\tau$.

By set differences, then, the result is proved.
\end{proof}

It can be shown similarly that
$$\AVinP{2,13}{2/4} =
\MeshPattern{}{4}{1/2, 2/4, 3/1, 4/3}{1/0, 1/1, 1/2, 1/3, 1/4, 3/0, 3/1, 3/2, 3/3, 3/4}
-\MeshPattern{}{4}{1/2, 2/4, 3/1, 4/3}{0/3, 0/4, 1/0, 1/1, 1/2, 1/3, 1/4, 3/0, 3/1, 3/2, 3/3, 3/4}$$
These equivalences allow for mesh pattern formulations of
Theorem~\ref{thm:depth of perm} and~\ref{thm:length of perm}.

\section{Shallow Cycles and Involutions}\label{sec:shallow cycles and involutions}

In this section, we specialize our interest in shallow permutations to shallow cycles and shallow involutions.

Integer sequences celebrated throughout combinatorics turn up in the literature
related to the length, depth, and reflection-length permutation statistics.  For
example in \cite{petersen tenner}, Petersen and the second author show that the
number of permutations $\sigma \in \symm_n$ which have $\length(\sigma) =
\reflength(\sigma)$ is given by the Fibonacci numbers \cite[A000045]{oeis}, and
it is shown in \cite{diaconis graham} that the cases where $\length(\sigma) =
\depth(\sigma)$ are given by the Catalan numbers \cite[A000108]{oeis}.  

Here we will show that the shallow involutions are enumerated by the
Motzkin numbers \cite[A001006]{oeis}, via a bijection between shallow
involutions and circles with non-intersecting chords.  Then we classify shallow
cycles and show they are enumerated by Schr\"oder numbers \cite[A006318]{oeis},
via a bijection with separable permutations. 

\begin{corollary}\label{cor:motzkin involutions}
An involution $\sigma$ is shallow if and only if $\Fundamental{\sigma}$ avoids $\VinPat{31,42}$, and shallow involutions are counted by the Motzkin numbers.
\end{corollary}

\begin{proof}
Let $\sigma$ be an involution, meaning that its cycle form consists of $1$- and $2$-element cycles. Thus in $\Fundamental{\sigma}$, its image under the fundamental bijection, the descents occur between letters in the same $2$-cycle of $\sigma$, and all other values are fixed points of $\sigma$. 
Moreover, the fundamental bijection $\Fundamental{\sigma}$ 
will have no $\AVinP{2,13}{2/4}$ patterns because
after the ``$4$," we cannot have a consecutive ``$\VinPat{13}$."
Thus an involution $\sigma$ is shallow if and only if $\Fundamental{\sigma}$
avoids the pattern $\VinPat{31,42}$.

Motzkin numbers count the number of ways to draw non-intersecting chords on a
circle with $n$ points.
We can represent an involution $\sigma \in \symm_n$ as $n$ points on a circle
with chords connecting points in the same cycle in $\sigma$.
To say that $\Fundamental{\sigma}$ avoids $\VinPat{31,42}$ is equivalent to
saying that these chords are non-intersecting.
\end{proof}

We demonstrate this correspondence with two examples.

\begin{example}\label{example 4.10}
The involution $(32)(4)(51)(7)(86) = 53241876$ is shallow: its depth is $7$, its length is $11$, its reflection length is $3$, and $7 = (11+3)/2$.
The involution $(32)(4)(61)(7)(85) = 63248175$ produces 
$\VinPat{61,85} \sim \VinPat{31,42}$
in $\Fundamental{63248175} = 32461785$, and so it is not shallow. Indeed, its depth is $9$, its length is $13$, and its reflection length is $3$. As we see in Figure~$\ref{img1}$, the diagram for the first permutation (on the left) has no crossed chords, while the diagram for the second permutation (on the right) has a crossing.
\begin{figure}[htbp]
\begin{tikzpicture}
\draw (0,0) circle (1.5cm);
\node[draw=none,minimum size=3cm,regular polygon,regular polygon sides=8] (a) {};
\foreach \x in {1,...,8} {\fill (a.corner \x) circle (2pt);}
\draw (a.corner 1) -- (a.corner 5);
\draw (a.corner 6) -- (a.corner 8);
\draw (a.corner 2) -- (a.corner 3);
\node[draw=none,minimum size=3.75cm,regular polygon,regular polygon sides=8] (b) {};
\foreach \x in {1,...,8} {\draw (b.corner \x) node {$\x$};}
\end{tikzpicture}
\hspace{.5in}
\begin{tikzpicture}
\draw (0,0) circle (1.5cm);
\node[draw=none,minimum size=3cm,regular polygon,regular polygon sides=8] (a) {};
\foreach \x in {1,...,8} {\fill (a.corner \x) circle (2pt);}
\draw (a.corner 1) -- (a.corner 6);
\draw (a.corner 5) -- (a.corner 8);
\draw (a.corner 2) -- (a.corner 3);
\node[draw=none,minimum size=3.75cm,regular polygon,regular polygon sides=8] (b) {};
\foreach \x in {1,...,8} {\draw (b.corner \x) node {$\x$};}
\end{tikzpicture}
\caption{Involutions represented as circles with chords as discussed in Example~\ref{example 4.10}.
The involution on the left (with no crossed chords) is shallow, while the one on the right
(with crossed chords) is not.}\label{img1}
\end{figure}
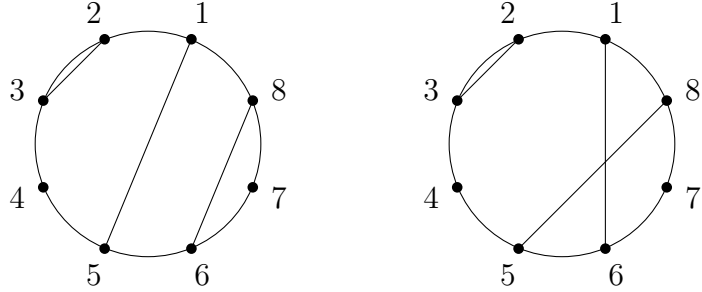
\end{example}

Having characterized and enumerated shallow involutions so succinctly, we now turn our attention to another special class of permutations: shallow cycles. To begin, we make some basic observations.

\begin{observation}\label{obs:simplify cycle pattern}
For any cycle $\sigma$,
$$\PatternCount{\AVinP{1,23}{1/4}}{\Fundamental{\sigma}} = \PatternCount{\VinPat{14,23}}{\Fundamental{\sigma}}
\text{ \ and \ }
\PatternCount{\AVinP{2,13}{2/4}}{\Fundamental{\sigma}} = \PatternCount{\VinPat{24,13}}{\Fundamental{\sigma}}.$$
\end{observation}

This allows for considerable simplification of the results from the previous section.

\begin{corollary}\label{cor:results for cycles}
Let $\sigma \in \symm_n$ be a cycle. Then
\begin{align*}
\length(\sigma) &= n-1+ 2\left(\VinCount{2,31} + \VinCount{14,23} +\VinCount{41,32}\right)(\Fundamental{\sigma}) \text{ and}\\
\depth(\sigma) &= n-1 + \left(\VinCount{2,31}
+ \VinCount{14,23} + \VinCount{41,32} + \VinCount{24,13} + \VinCount{31,42}\right)(\Fundamental{\sigma}),\\
&= \frac{\length(\sigma) + \reflength(\sigma)}{2} + \left(\VinCount{31,42} + \VinCount{24,13}\right)(\Fundamental{\sigma}).
\end{align*}
In particular, a cycle $\tau$ is shallow if and only if $\Fundamental{\tau}$ avoids
$\VinPat{31,42}$ and $\VinPat{24,13}$.
\end{corollary}
\begin{proof}
  The sum of a permutation's reflection length and number of cycles is equal to its size. Let $\sigma \in \symm_n$ be a cycle, and so $\reflength(\sigma) = n-1$.
  By Equation~\eqref{eqn: length formula} and Observation~\ref{obs:simplify cycle pattern} we have $\length(\sigma) = n-1+ 2\left(\VinCount{2,31} + \VinCount{14,23} +\VinCount{41,32}\right)(\Fundamental{\sigma})$. Thus
  $$\frac{\length(\sigma) + \reflength(\sigma)}{2} = n-1+ \left(\VinCount{2,31} + \VinCount{14,23} +\VinCount{41,32}\right)(\Fundamental{\sigma}).$$
Using this, together with Theorem~\ref{thm:depth of perm} and Observation~\ref{obs:simplify cycle pattern}, we have
  \begin{align*}
  \depth(\sigma) &= n-1 + \left(\VinCount{2,31} + \VinCount{14,23} + \VinCount{41,32} +
  \VinCount{24,13} + \VinCount{31,42}\right)(\Fundamental{\sigma})\\
&= \frac{\length(\sigma) + \reflength(\sigma)}{2} + \left(\VinCount{31,42} + \VinCount{24,13}\right)(\Fundamental{\sigma}).
\end{align*}
\end{proof}

To make a better characterization of shallow cycles, we note the
following pattern coincidence. 

\begin{definition}\label{defn:coincidence}
For two sets of patterns, $A$ and $B$, write $A \asymp B$ to indicate that $A$-avoidance is equivalent to $B$-avoidance; that is, that a permutation avoids all elements of $A$ if and only if that permutation avoids all elements of $B$. When $A \asymp B$, we say that $A$ and $B$ are \emph{coincident}.
\end{definition}

\begin{lemma}\label{lem:3142 and 2413}
$\{3142, 2413\} \asymp \{\VinPat{31,42}, \VinPat{24,13}\}$.
\end{lemma}

\begin{proof}
We first show that $P:= \{3142, 2413\}$ is coincident to $Q:=\{\VinPat{31,4,2}, \VinPat{24,1,3}\}$. Certainly if a permutation avoids all $P$-patterns then it also avoids all $Q$-patterns. Now suppose that $\sigma$ avoids all $Q$-patterns, but has a $3142$-pattern in positions $i_1 < i_2 < i_3 < i_4$. Then it must be that $i_2 > i_1 + 1$. Let us choose $i_1$ to be maximal relative to $i_2$, and consider $\sigma_{i_1+1}$. To avoid a $\VinPat{31,4,2}$-pattern, we must have $\sigma_{i_1+1} > \sigma_{i_4}$. By maximality of $i_1$, we must also have $\sigma_{i_1+1} > \sigma_{i_3}$. But then $\sigma$ has a $\VinPat{24,1,3}$-pattern in positions $\{i_1, i_1+1, i_2, i_3\}$, which is a contradiction. Thus $\sigma$ must avoid $3142$, and a similar argument shows that it avoids $2413$. Hence $P \asymp Q$.

We will now show that $Q$ is coincident to $R := \{\VinPat{31,42}, \VinPat{24,13}\}$, and the result will follow by transitivity of $\asymp$.

As before, if $\sigma$ avoids all $Q$-patterns, then it necessarily avoids all $R$ patterns. Now consider $\sigma$ avoiding all $R$-patterns and suppose that $\sigma$ has a $\VinPat{31,4,2}$-pattern in positions $i_1 < i_2 < i_3 < i_4$, with $i_2 = i_1+1$. Then we must have $i_4 > i_3 + 1$. Choose this occurrence of $\VinPat{31,4,2}$ so that $i_4 - i_3$ is minimal. Because of this minimality, and because $\sigma$ avoids $\VinPat{31,42}$, we must have $\sigma_j < \sigma_2$ for $j \in \{i_3 + 1, i_4 - 1\}$. In other words, there is a $\VinPat{42,51,3}$-pattern in positions $\{i_1, i_2, i_3, i_3+1, i_4\}$ of $\sigma$, and a $\VinPat{42,5,13}$-pattern in positions $\{i_1, i_2, i_3, i_4-1, i_4\}$. Because $\sigma$ avoids all $R$-patterns, we must have $\sigma_{i_3+1} < \sigma_{i_2+1} < \sigma_{i_4}$. If $\sigma_{i_2+1} < \sigma_{i_4-1}$, then $\sigma$ will have a $\VinPat{31,4,2}$-pattern in positions $\{i_2,i_2+1,i_3,i_4-1\}$, violating minimality of $i_4-i_3$. Thus $\sigma_{i_4-1} < \sigma_{i_2+1}$. In fact, a similar analysis shows that $\max\{\sigma_{i_3+1}, \sigma_{i_4-1}\} < \sigma_j < \sigma_{i_4}$ for all $j \in (i_2,i_3)$. But then $\sigma$ has a $\VinPat{24,13}$-pattern in positions $\{i_3-1, i_3, i_4-1, i_4\}$, which is a contradiction. Thus $\sigma$ must avoid $\VinPat{31,4,2}$, and a similar argument shows that it avoids $\VinPat{24,1,3}$. Hence $Q \asymp R$.
\end{proof}

We can now characterize and enumerate shallow cycle permutations. An alternative proof of this result may also be deduced from \cite[Theorem 1.1]{cornwell mcnew} together with \cite[Theorem 1.1]{woo}, the latter of which in preparation when our work was first posted to the arXiv.

\begin{corollary}\label{cor:count shallow cycles}
There is a bijection between shallow cycle permutations in $\symm_n$ and separable permutations in $\symm_{n-1}$, and they are enumerated by the $(n-1)$st large Schr\"{o}der number.
\end{corollary}
\begin{proof}
Let $\sigma \in \symm_n$ be a cycle. Corollary~\ref{cor:results for cycles} says that $\sigma$ is shallow if and only if $(\VinCount{31,42} + \VinCount{24,13})(\Fundamental{\sigma}) = 0$. By Lemma~\ref{lem:3142 and 2413}, this is equivalent to $\Fundamental{\sigma}$ avoiding both $3142$ and $2413$, meaning that $\Fundamental{\sigma}$ is \emph{separable}. 
Because $\sigma$ is a cycle, the leftmost letter in the one-line notation for $\Fundamental{\sigma}$ is $n$, and the rest of the word can be any separable permutation in $\symm_{n-1}$. These are enumerated by the large Schr\"{o}der numbers \cite[A006318]{oeis}.
\end{proof}

It is worth noting that the proof of this corollary goes beyond counting shallow $n$-cycles
to clearly describe their structure.
Let $\tau$ be the cycle $(n\ 1\ 2\ \cdots\ n-1)$ in $\symm_n$.
A cycle $\sigma \in \symm_n$ is shallow if and only if there exists a separable permutation
$\pi \in \symm_{n-1}$ such that $\sigma = \pi^{-1} \tau \pi$.

\section{Directions for further research}\label{sec:future}

The goal of this paper was to explore pattern-functions for permutation statistics and to demonstrate their utility by characterizing shallow permutations. Given the success of this approach, it seems likely that other quantities can also be written as pattern-functions, and perhaps other open questions can be similarly resolved. We have also used pattern-functions in this work to enumerate interesting classes of permutations. Those efforts, too, suggest a broader category of problems that can benefit from the perspective of pattern-functions.

\section*{Acknowledgements}

We are grateful for thoughtful comments and suggestions from the anonymous referees.

\end{document}